\documentclass[a4paper,12pt]{scrartcl}
\usepackage[english]{babel}
\usepackage[T1]{fontenc}
\usepackage[utf8]{inputenc}
\DeclareUnicodeCharacter{00A0}{ }
\DeclareUnicodeCharacter{014C}{\=O}
\usepackage{mathptmx}
\usepackage[scaled=.92]{helvet}
\usepackage{courier}
\usepackage{verbatim}
\usepackage{version}
\usepackage{paralist}
\usepackage[nodetail]{optional}

\usepackage[fixlanguage]{babelbib}
\selectbiblanguage{english}

\usepackage{amsmath,amsfonts,amssymb}

\newcommand{\N}{\mathbb{N}}
\newcommand{\Z}{\mathbb{Z}}

\newcommand{\R}{\mathbb{R}}
\newcommand{\C}{\mathbb{C}}
\newcommand{\RE}{\operatorname{Re}}
\newcommand{\IM}{\operatorname{Im}}
\newcommand{\Res}{\operatorname{Res}}
\newcommand{\eps}{\varepsilon}
\newcommand{\bigoh}{\mathcal{O}}

\newcommand{\diag}{\operatorname{diag}}
\newcommand{\intd}{\mspace{0.3mu}\operatorname{d}\mspace{-2.0mu}}
\newcommand{\supp}{\operatorname{supp}}

\newcommand{\Id}{\operatorname{Id}}

\newcommand{\Cinfty}{\ensuremath{C^{\infty}}}
\newcommand{\Ccinfty}{\ensuremath{\Cinfty_c}}

\newcommand{\Comega}{C^{\omega}} 

\newcommand{\Dprime}{{\ensuremath{\mathcal{D}^\prime}}}
\newcommand{\Dprimebar}{\bar\Dprime}
\newcommand{\Eprime}{\ensuremath{\mathcal{E}^\prime}}

\newcommand{\T}{\ensuremath{\mathcal{T}}}

\newcommand{\lie}{\mathfrak}
\newcommand{\halb}{{\textstyle\frac{1}{2}}}
\newcommand{\bv}{\beta}  

\newcommand{\ad}{\operatorname{ad}}
\renewcommand{\a}{\lie{a}}  
\newcommand{\aC}{\lie{a}^*_{\C}}

\newcommand{\D}{{\mathbb D}}
\newcommand{\E}{E}
\newcommand{\Et}{\bar\E}
\newcommand{\LinOp}{{\mathcal L}} 
\newcommand{\Aprime}{{\mathcal A}'} 
\newcommand{\sfcn}{\phi}  
\newcommand{\cfcn}{\mathbf{c}} 
\newcommand{\ksq}[1]{\langle#1,#1\rangle}
\newcommand{\laplaceop}{\Delta_X}

\newcommand{\Lloc}{L_{\operatorname{loc}}}
\newcommand{\Sob}[1]{H^{#1}} 

\newcommand{\SobDot}[1]{\dot H^{#1}}

\newcommand{\Rn}[1]{{\romannumeral #1}}
\newcommand{\RN}[1]{\uppercase\expandafter{\romannumeral #1\relax}}

\usepackage{amsthm}
\newtheorem{theorem}{Theorem}[section]

\newtheorem{proposition}[theorem]{Proposition}

\newtheorem{corollary}[theorem]{Corollary}
\newtheorem{lemma}[theorem]{Lemma}
\theoremstyle{definition}

\newtheorem{remark}[theorem]{Remark}

\usepackage[pdftex]{color}
\usepackage[pdftex]{graphicx}
\usepackage[pdftex,
	colorlinks=true,
	linkcolor=gruen,
        urlcolor=blue,
        citecolor=blue,
	bookmarksopen=false,
	pdfpagemode=None]{hyperref}
\definecolor{gruen}{rgb}{0,0.5,0}

\title{Resonances and Scattering Poles
  \\ in Symmetric Spaces of Rank One}
\author{S.~Hansen, J.~Hilgert, and A.~Parthasarathy}
\makeindex
\begin{document}

\maketitle

\begin{abstract}
We relate resolvent and scattering kernels for the Laplace operator on Riemannian symmetric spaces of rank one via boundary values in the sense of Kashiwara-{\=O}shima.
From this, we derive that the poles of the corresponding meromorphic continuations agree in a half-plane,
and the residues correspond to each other under the boundary value map, so in particular the multiplicities agree as well.
In the opposite half-plane, which is the square root of the resolvent set,
the resolvent has no poles, whereas the scattering poles agree with the poles of the standard Knapp--Stein intertwiner.
As a by-product of the underlying ideas, we obtain a new and self-contained proof of Helgason's conjecture for distributions in the case of rank one symmetric spaces.
\end{abstract}

\section{Introduction}

The close relation between resonances and scattering poles was first established rigorously for  surfaces with hyperbolic ends by L. Guillop\'e and M. Zworski in \cite{GuiZworski97riesurf}. 
Their results were extended by D. Borthwick and P. Perry in \cite{BorthwPerry01} and C. Guillarmou in \cite{Guillarmou05reso,GuillaBarreto08ach} to the case of asymptotically real (and in part also complex) hyperbolic manifolds. The prime examples of such manifolds are locally symmetric spaces given as quotients of real, respectively complex, hyperbolic spaces by convex cocompact discrete subgroups. In this note we treat the problem in the case of arbitrary noncompact Riemannian symmetric spaces of rank one. This means we deal with real, complex, quaternionic and octonionic hyperbolic spaces (which exists only in dimension three over the octonions).  Extensions to locally symmetric spaces will be addressed elsewhere.  

Let $\laplaceop$ be the positive Laplacian on a Riemannian
symmetric space $X=G/K$ of rank one. We consider a modified resolvent $R_\zeta$, where $\zeta$ is a square root of the spectral parameter. For each $\zeta$ the Schwartz kernel of $R_\zeta$ satisfies, away from the diagonal, the Laplace eigenequation in each variable. It is well-known \cite{MiatelloWill99resid,HilgertPasquale09reson} that as a family of distributions it admits a meromorphic continuation to the entire complex plane with poles, called resonances, of finite rank.     

The symmetric space $X$ can be realized as an open ball with a natural boundary $B$ which is diffeomorphic to a sphere. The maximal compact subgroup $K$ of $G$ acts transitively on $B$, so it can be written as $B=K/M$. In fact, the action extends to an action of $G$ with stabilizer $P=MAN$, a (minimal) parabolic subgroup of $G$.
The asymptotic behaviour of Laplace eigenfunctions at $B$ is understood by geometric scattering theory on $X$.
In \cite{SemenTianSanskii76}, a Lax--Phillips type scattering theory for multi-temporal wave equations was developed,
and it was shown that the corresponding scattering matrices basically coincide with the intertwining operators between spherical principal series representations as introduced by Helgason, Schiffmann and Knapp-Stein. These intertwining operators depend meromorphically on a complex parameter (see e.g. \cite{Knapp86repexa}). 

The starting point for the identification of resonances and scattering poles in \cite{BorthwPerry01} and \cite{Guillarmou05reso} is the observation from \cite{JoshiSaBarreto00invscatt} that for asymptotically hyperbolic manifolds the scattering matrix can be written as a boundary value, in the framework of the Mazzeo-Melrose $0$-calculus, of the resolvent kernel. In contrast, for symmetric spaces of rank one, our key observation is that a slight extension of
the theory of boundary values of eigenfunctions for invariant differential operators allows us to establish a similar relation by considering boundary values of the resolvent kernels.

More precisely, recall that functions (more generally, hyperfunctions) on $B$ are mapped to Laplace eigenfunctions via the Poisson transform, the eigenvalues and the Poisson transform being paramatrized by the same complex parameter. For generic parameters, the Poisson transform is invertible, which, in rank one, was first proved in \cite{Lewis78eigenf}. In higher rank, the corresponding statement became known as  Helgason's conjecture. This conjecture was proved in \cite{Kashiwara78eigen} by showing that for generic parameters there is a boundary value map on the space of Laplace eigenfunctions which inverts the Poisson transform. The proof uses Sato's theory of hyperfunctions and the microlocal theory of boundary values for partial differential operators with regular singularities as established in \cite{KashiwaraOshima77}. In an appendix, we give a new and self-contained proof of the existence of a boundary value map  for rank one symmetric spaces, following the simplified approach from \cite{Oshima83bdryval} closely. This yields a proof of Helgason's conjecture (at the level of distributions) using only techniques from partial differential equations. Our proof does not involve the rather restrictive assumption $(A)$ of \cite{Kashiwara78eigen},
but rather their assumption $(A')$. Moreover, it is an important and useful feature of our approach that Fatou-type theorems 
can be used directly to calculate boundary values. As a result, we obtain a version of the boundary value map that can be applied to each argument of the resolvent kernel. In this way, we can not only find the Poisson kernel as a boundary value of the resolvent kernel, but also the scattering matrix.
This gives the equality of resonances and scattering poles in the halfplane $\IM\zeta>0$,
and also enables us to establish a bijection between the ranges of the respective residues.
The (physical) halfplane $\IM\zeta < 0$ contains no resonances.
However, the scattering poles in this half-plane agree with poles of the Knapp--Stein intertwiner.

We limit our considerations to the rank one case because in higher rank it is not clear if
the resolvent kernel satisfies, away from the diagonal, the joint eigenequations
of the algebra of invariant differential operators.
This would be necessary to carry over our construction of boundary values of the resolvent kernel.

We now describe the organization of the paper briefly. In Sections~\ref{sect-Poisson}
and \ref{sect-resolvent}
we summarize the features of the Poisson transform and the meromorphically extended resolvent kernel that we need in the sequel. In Section~\ref{sect-spectral} we compute the spectral density of the shifted Laplacian in terms of the Poisson transform. In view of Stone's formula, this boils down to Theorem~\ref{thm-resolvent-poisson} which relates the resolvent to the Poisson transform and is based on the asymptotic expansion \eqref{R-asymptotics} of the resolvent kernel at the boundary. In Section~\ref{sect-Resonances} we recall the
location of the resonances and the corresponding residues. In fact, we use the results from Section~\ref{sect-spectral} to give short proofs for most of the  facts mentioned, including a new formula for the residues. An exception is the finite rank property, obtained using highest weight theory, for which we simply cite  \cite{MiatelloWill99resid,HilgertPasquale09reson}. 
In Section~\ref{sect-Boundary} we prove that the Poisson kernel can be written as a boundary value of the resolvent kernel,
Corollary~\ref{cor-poisson-is-by-of-resolvent}.
The technical heart of this result is Lemma~\ref{lemma-bv-def-and-limit} on the existence of boundary value maps,
for which we give a self-contained proof in Appendix~\ref{appx-bvmap}. Moreover, Section~\ref{sect-Boundary} contains the new proof of Helgason's conjecture in rank one mentioned above. See Theorem~\ref{Thm-KKMOOT} for the precise formulation.
In Section~\ref{sect-Scatt} we define the scattering matrix as an operator relating different boundary values.
We give formulae relating the scattering matrix, the resolvent, the Poisson transform, the Harish-Chandra $\cfcn$-function,
and the standard Knapp-Stein intertwiner.
Theorem~\ref{thm-scatt-poles-are-resonances}, which is our main result, states that a scattering pole
is either a resonance or a pole of the standard intertwiner, and that the resolvent residue is mapped isomorphically onto the corresponding residue of the scattering operator by the boundary value map.

\bigskip

\noindent{\textbf{Acknowlegdements.}} We thank David Borthwick, Colin Guillarmou, Angela Pasquale and Tobias Weich for helpful conversations on the subject of this article.

\section{Poisson Transformation}
\label{sect-Poisson}

Let $X=G/K$ be a Riemannian symmetric space of noncompact type with isometry group $G$,
$G$ a connected semisimple Lie group with finite centre, and isotropy subgroup $K$ which is maximal compact.
Denote the algebra of invariant differential operators on $X$ by $\D(X)$.
The joint eigenspace associated with a character $\chi:\D(X)\to\C$
is the space $\E_\chi(X)$ consisting of functions $u$ which satisfy $Du=\chi(D)u$, $D\in\D(X)$.
Since $\D(X)$ contains an elliptic operator, the Laplacian $\laplaceop$,
$\E_\chi(X)$ is a subspace of the space $\Comega(X)$ of real analytic functions on $X$.
Left translation by elements of $G$ defines the eigenspace representation $T_\lambda$ of $G$ on $\E_\chi(X)$.
Denote by $d_X(x)$ the distance from $x\in X$ to the origin $o=K\in X$.

Fix an Iwasawa decomposition $G=KAN$.
Thus the following data are chosen:
A Cartan decomposition $\lie{g}=\lie{k}+\lie{p}$ of the Lie algebra of $G$,
a maximal abelian subalgebra $\a$ of $\lie p$,
and a positive system $\Sigma^+\subset\a^*$ of the set of restricted roots.
Denote by $\rho\in\a^*$ the weighted half-sum of positive roots.
The Iwasawa projections $\kappa:G\to K$ and $H:G\to \a$ are given by $g\in \kappa(g)\exp(H(g)) N$.

Let $M$ denote the centralizer in $K$ of the Lie algebra $\a$ of $A$.
The compact space $B=K/M$ is called the Furstenberg boundary of $X$.
Let $\a^+$ be the positive Weyl chamber in $\a$ determined by $\Sigma^+$, and 
$\a^*_+$ the dual cone in $\a^*$ of $\a^+$.
Set $A^+=\exp\a^+$.
By the $KAK$ decomposition, $G=KAK$, we have the diffeomorphism $(kM,a)\mapsto ka\cdot o$
from $B\times A^+$ onto the open dense subset $X'=KA^+\cdot o$ of $X$.

Let $W$ denote the Weyl group for the restricted roots, and $\aC$ the complexified dual of $\a$.
The Harish-Chandra isomorphism is an algebra isomorphism  $\Gamma$ from $\D(X)$
onto the algebra of $W$-invariant polynomials on $\aC$.
Every character $\chi$ of $\D(X)$ is of the form $\chi(D)=\Gamma(D)(\lambda)$ for some $\lambda\in\aC$.
We then set $\E_\lambda(X)=\E_\chi(X)$.
Notice that $\E_{w\cdot\lambda}(X)= \E_{\lambda}(X)$ for $w\in W$.

The homogeneous space $G/P$, $P=MAN$ minimal parabolic,
is diffeomorphic to $B$ by the map $kM\mapsto kP$, $k\in K$.
This defines a $G$-action on $B$, $g\cdot kM=\kappa(gk)M$.

We set $a^\mu=e^{\mu(\log a)}$ if $a\in A$, $\mu\in\aC$.
Every $\lambda$ in $\aC$ defines a character of $P$, $man\mapsto a^{\rho-\lambda}$.
Denote by $L_\lambda\to G/P$ the associated $G$-homogeneous line bundle.
Inducing from $P$ to $G$, the spherical principal series representation $\pi_\lambda$ of $G$
in the space $\Aprime(G/P; L_\lambda)$ of analytic functionals (hyperfunctions) is obtained.
Sections are canonically identified with complex-valued functions $\varphi$ on $G$
which satisfy $\varphi(gman)=a^{\lambda-\rho}\varphi(g)$ for $g\in G$, $man\in P$.
This gives the induced picture of $\pi_\lambda$ in the sense of \cite[Ch.~\RN 7 \S 1]{Knapp86repexa}.
Restriction of $\varphi$ to $K$ gives an isomorphism of locally convex spaces,
$\Aprime(G/P; L_\lambda) \equiv \Aprime(B)$.
The inverse assigns to a function $f$ on $B$ the section $\varphi$ with
$\varphi(kan)=a^{\lambda-\rho}f(kM)$, $a\in A$ and $n\in N$.
This isomorphism transfers $\pi_\lambda$ from the induced to the compact picture,
where the representation space of $\pi_\lambda$ is $\Aprime(B)$:
\begin{equation}
\label{eq-spher-rep-on-Aprime-B}
(\pi_\lambda(g) f)(kM) = e^{-(\rho-\lambda)(H(g^{-1}k))} f(\kappa(g^{-1}k)M).
\end{equation}
See \cite[Ch.~\RN 7 \S 1]{Knapp86repexa}, and \cite[Ch.~\RN 6 \S 3 (13)]{Helgason94GASS}.

The Poisson transformation $P_\lambda$ with spectral parameter $\lambda\in\aC$ is a $G$-homomorphism
which intertwines the spherical principal series representation and the eigenspace representation:
\[
P_\lambda: \Aprime(G/P; L_\lambda) \to \E_\lambda(X),
\quad
(P_\lambda f)(gK) = \int_K f(gk)\intd k.
\]
Alternatively, in the compact picture of $\pi_\lambda$, $P_\lambda:\Aprime(B)\to \E_\lambda(X)$,
\begin{equation}
\label{def-Poisson-transform}
(P_\lambda f)(x)= \int_B e^{(\rho+\lambda)(A(x,b))} f(b)\intd b.
\end{equation}
Here $A$ denotes the horocycle bracket $A(x,b)=-H(g^{-1}k)\in\a$ of $x=gK\in X$, $b=kM\in B$.
The spherical function $\sfcn_\lambda =P_\lambda 1$ is the unique $K$-invariant
element of $\E_\lambda(X)$ which satisfies $\sfcn(o)=1$.
Moreover, $\sfcn_{w\cdot\lambda}=\sfcn_\lambda$ holds for $w\in W$ and $\lambda\in\aC$.
To ease notation, we write
\[ P_\lambda(x,b)=  e^{(\rho+\lambda)(A(x,b))} \]
for the Schwartz kernel of $P_\lambda$.

Write $\bar N$ for the image of the nilpotent subgroup $N$  under the Cartan involution.
The Harish-Chandra $\cfcn$-function is the meromorphic function on $\aC$ given by
$\cfcn(\lambda)=\int_{\bar N}e^{-(\lambda+\rho)(H(\bar n))}\intd \bar n$
when $\RE\lambda\in \a_+^*$.
It arises, for example, in the following Fatou-type theorem,
\cite[Ch.~\RN 2, Theorem 3.16]{Helgason94GASS}:
If $\RE\lambda\in \a_+^*$, $H\in\a^+$, and $f\in C(B)$, then
\begin{equation}
\label{Helgason-Fatou-thm}
\lim_{t\to\infty} a_t^{\rho-\lambda} (P_\lambda f)(ka_t\cdot o)   = \cfcn(\lambda)f(kM),
\end{equation}
where $a_t=\exp(tH)$.
Moreover, the convergence is uniform on $B$.

Let $\{\alpha_1,\ldots,\alpha_\ell\}\subset \Sigma^+$ denote the simple positive roots,
$\ell$ the real rank of $X$.
Let $H_1,\ldots,H_\ell\in\a^+$ be the dual basis, that is $\alpha_j(H_k)=\delta_{jk}$.
We identify $\aC$ with $\C^\ell$ via the basis $\alpha_1,\ldots,\alpha_\ell$, so
$\zeta = \big(\zeta(H_1),\ldots, \zeta(H_\ell)\big) = \sum\nolimits _j \zeta_j\alpha_j$
in $\aC=\C^\ell$.
Moreover, following \cite{Kashiwara78eigen},
we identify $A^+$ with the open octant $]0,1[^\ell$ by the diffeomorphism
\[
O:=]0,1[^\ell \to A^+,\quad y=(y_1,\ldots,y_\ell)\mapsto a=\exp(\sum\nolimits _j t_jH_j),\quad y_j=e^{-t_j}.
\]
Using this, we identify $X'=KA^+\cdot o$ and $B\times O$ by
the real analytic diffeomorphism $ka\cdot o \equiv (kM,y)$.
If $s=(s_1,\ldots,s_\ell)\in\C^\ell$, then we write $y^s$ for the
function $(b,y)\mapsto y_1^{s_1}\cdots y_\ell^{s_\ell}$.

Let $g\in G$.
The map $x\mapsto g^{-1}\cdot x$ restricts to a diffeomorphism
from the open dense subset $X'\cap g\cdot X' \subset X$ onto $g^{-1}\cdot X'\cap X'$.
By Lemmas 4.1 and 4.2 of \cite{Kashiwara78eigen} the following holds:
There exists a real analytic diffeomorphism $\Phi_g:U\to U'$, $(b,y)\mapsto(b',y')$,
between open neighbourhoods $U,U'\subset B\times\widetilde O$, $\widetilde O:=]-1,1[^\ell$, of $B\times\{0\}$
such that $g^{-1}\cdot (b,y)=(b',y')$ for $(b,y)\in U\cap X'\cap g\cdot X'$, and
\begin{equation}
\label{Jacobian-yprime-y}
\Phi_g(kM,0) = (\kappa(g^{-1}k)M, 0), \quad
\frac{\partial y'}{\partial y}(b,0)
   =\diag\big(e^{\alpha_1(A(g\cdot o, b))}, \ldots, e^{\alpha_\ell(A(g\cdot o, b))}\big).
\end{equation}
Given $\mu=(\mu_1,\ldots,\mu_\ell)\in \C^\ell$ we have the line bundle $(N^*B)^\mu$,
where $N^*B\subset T^*(B\times\widetilde O)$ denotes the conormal bundle of $B\equiv B\times\{0\}$.
The bundle $(N^*B)^\mu$ is trivialized by the global section $(\intd y_1)^{\mu_1}\cdots(\intd y_\ell)^{\mu_\ell}$.
It follows from \eqref{eq-spher-rep-on-Aprime-B} and \eqref{Jacobian-yprime-y} that
$(N^*B)^{\rho-\lambda}$ is, as a $G$-equivariant line bundle, isomorphic to $L_\lambda$,
when the latter is pulled back from $G/P$ to $B$.
See \cite[Lemma 5.2.1]{Schlichtkrull84hyperf}.
The $G$-isomorphisms
\begin{equation*}
\Aprime(G/P;L_{\lambda})\equiv\Aprime(B;(N^*B)^{\rho-\lambda})\equiv\Aprime(B)
\end{equation*}
relate different realizations of the spherical principal series representation.
Each realization is intertwined with the eigenspace representation by the Poisson transformation.

We shall work in a distributional setting rather than with hyperfunctions.
If $f\in\Dprime(B)$ is a distribution, then $u=P_\lambda f$ is of weak moderate growth,
which means that $ue^{-rd_X}$ is a bounded function for some $r>0$, \cite[Section~I.2]{BanSchlichtkrull87}.
Equivalently, there exists $\nu\in\N^\ell$ such that $y^\nu u$ is bounded on $B\times O$;
see \cite[Lemma 2.1 (\Rn{3}), (\Rn{4})]{BanSchlichtkrull87}.
In this paper, we handle the growth condition using Sobolev spaces of negative order.

The range of the restriction map $\Dprime(B\times\widetilde O)\to\Dprime(B\times O)$
is, by definition, the space $\Dprimebar(B\times O)$ of extendible distributions.
\begin{lemma}
Functions of weak moderate growth are extendible distributions.
\end{lemma}
\begin{proof}
Consider the manifold $B\times \R^\ell$ with its product Riemannian structure,
not to be confused with the Riemannian structure induced from $X$.
The set $B\times O$ is a relatively compact open subset of
$B\times \R^\ell$, and it has a Lipschitz boundary.
We denote by $\Sob{m}(B\times O)\subset\Dprimebar(B\times O)$ the usual Sobolev space of order $m\in\Z$.
In particular, $\Sob{0}(B\times O)=L^2(B\times O)$.
Furthermore, we denote by $\SobDot{m}(B\times \bar O)$ the subspace of $\Sob{m}(B\times\R^\ell)$
which consists of elements supported in $B\times\bar O$, $\bar O$ the closure of $O\subset\R^\ell$.
If $m\geq 0$ then $\Ccinfty(B\times O)$ is dense in $\SobDot{m}(B\times \bar O)$,
and $\SobDot{m}(B\times \bar O)\subset\Sob{m}(B\times O)$.
Recall that $\Sob{-m}(B\times O)$ is the dual of $\SobDot{m}(B\times \bar O)$.
If $0\leq k<m-\dim(X)/2$, then, by the Sobolev Lemma,
$\SobDot{m}(B\times \bar O)\subset \dot C^k(B\times\bar O)$,
the space of $C^k$-functions in $B\times\R^\ell$ which are supported in $B\times \bar O$.
By Taylor's formula, if $\nu\in\N^\ell$, $|\nu|\leq k$, then
$\dot C^k(B\times\bar O)\subset y^{\nu}\dot C^{k-|\nu|}(B\times\bar O)$.
It follows that multiplication by $y^{-\nu}$ maps 
$\SobDot{m}(B\times \bar O)$ continuously into $\SobDot{m'}(B\times \bar O)$,
and, dually,
\[ y^{-\nu}\Sob{-m'}(B\times O)\subset \Sob{-m}(B\times O) \quad\text{if $0\leq m'<m-\dim(X)/2-|\nu|$.} \]
Therefore, if $u$ is of weak moderate growth, then
$u\in\Sob{-m}(B\times O)$ for some $m\in\N$.
\end{proof}

Observe that $\E_\lambda(X)\cap\Sob{-m}(B\times O)$ is a closed subspace of
$\Sob{-m}(B\times O)$ because differential operators are continuous maps into the space of distributions.
The space $\Dprime(B)=\cup_m \Sob{-m}(B)$, and the space
\[
\Et_\lambda(X) = \E_\lambda(X)\cap \Dprimebar(B\times O)
    = \cup_{m\in\N} \big(\E_\lambda(X)\cap\Sob{-m}(B\times O)\big);
\]
of tempered (weak moderate growth) joint eigenfunctions are (DFS)-spaces,
that is they are strong duals of Fr\'echet--Schwartz spaces;
\cite[Prop.~25.20]{MeiseVogt97fa}.
The Poisson transform is a continuous linear operator $P_\lambda:\Dprime(B)\to\Et_\lambda(X)$.
We shall regard $P_\lambda$ also, by composing it with inclusion $\Et_\lambda(X)\subset\Cinfty(X)$,
as a map into $\Cinfty(X)$.
For every $m\in\N$ and every bounded domain $\Lambda\subset\aC$ there exists $m'\in\N$
such that $\Lambda\ni\lambda\to P_\lambda$ is a holomorphic map into the
Banach space $\LinOp\big(\Sob{-m}(B), \Sob{-m'}(B\times O)\big)$ of bounded linear operators.

\section{Resolvent}
\label{sect-resolvent}

From now on we assume that $X$ has rank one, $\ell=1$.
Then $\D(X)$ is generated by the Laplace operator $\laplaceop$ which is positive,
in fact, $\laplaceop\geq\ksq{\rho}$.
Since $X$ is a complete Riemannian manifold, $\laplaceop$ is essentially selfadjoint on $\Ccinfty(X)$.
The unique selfadjoint extension is also denoted $\laplaceop$.
By spectral theory, the (modified) resolvent
\begin{equation*}
R_\zeta = (\laplaceop-\ksq{\rho}-\ksq{\zeta})^{-1}
\end{equation*}
is a holomorphic map from the (physical) half-plane $\IM\zeta<0$
into the space of bounded linear operators on $L^2(X)$.
Here, the angular brackets denote the Killing form.
In \cite{MiatelloWallach92resolv} a continuation of $R_\zeta$ as a meromorphic
function into the space of linear operator from $L^2_c(X)$ into $\Lloc^2(X)$ is constructed.
Below, we revisit, and in part simplify, this construction.
See \cite{HilgertPasquale09reson} for a different approach.
The meromorphically continued resolvent will also be denoted $R_\zeta$,
and we also call it the resolvent.

The set $\Sigma^+$ of positive roots equals $\{\alpha\}$ or $\{\alpha,2\alpha\}$.
The multiplicities of the roots are $m_\alpha>0$ and $m_{2\alpha}$;
the latter is set to zero if $2\alpha\notin\Sigma^+$.
The dimension of the symmetric space $X$ is $\dim X=1+m_\alpha+m_{2\alpha}$.
Fix $H\in\a^+$ such that $\alpha(H)=1$.
The map
$\zeta\mapsto \zeta(H) = \langle \zeta,\alpha\rangle/\ksq{\alpha}$,
and its inverse $z\mapsto z\alpha$ identify $\aC$ with $\C$.
So $\ksq{\zeta}=\ksq{\alpha} \zeta^2$ and $\rho=(m_\alpha+2m_{2\alpha})/2\in\R$.
Set $\|\alpha\|= \ksq{\alpha}^{1/2}$.
Notice that $\|\alpha\| H$ is a unit vector, and that $A_\alpha=\ksq{\alpha} H$ is the vector in $\a$
which corresponds to $\alpha$ under the identification $\a^*\equiv \a$ defined by the Killing form.
Except for the duality bracket and the Killing form, operations with elements in $\aC$ are done in $\C$.
The Weyl group $W$ consists of the elements $\pm 1$.

The maps $t\mapsto tH$ and $t\mapsto \exp tH\cdot o$ identify $\R$ with $\a$
and $\R_+$ with $A^+\cdot o$, respectively.
Introducing the variable $y=e^{-t}$ we identify $A^+$ with the open unit interval $]0,1[$,
and therefore,
\begin{equation}
\label{X-ident-KM-unitinterval}
X'\equiv B\times ]0,1[, \quad x=ka\cdot o \equiv (kM,y)=(b,y), \quad \log a = -(\log y)H,
\end{equation}
where $X'=KA^+\cdot o=X\setminus\{o\}$.
The coordinate $y$ is a defining function of the boundary $B$
in the boundary collar $B\times]-1,1[$.
We identify distributions and differential operators on $X'$
with their pullbacks to $B\times]0,1[$.
Set
\begin{equation*}
J(y) =y^{-2\rho}(1-y^2)^{m_\alpha}(1-y^4)^{m_{2\alpha}}
   = (2\sinh t)^{m_\alpha} (2\sinh 2t)^{m_{2\alpha}}, \quad y=e^{-t}.
\end{equation*}
The measures $\intd x$ on $X$ and $\intd b$ on $B$ are related as follows:
\begin{equation} 
\label{eq-measure-dx} 
\intd x = J(y) y^{-1}\intd y \intd b
\end{equation} 
on $B\times]0,1[$; see \cite[Ch.~\RN 1, Theorem~5.8]{Helgason84GGA}.
Pullback by the identification map \eqref{X-ident-KM-unitinterval} gives a unitary isomorphism
\[ L^2(X)\equiv L^2(B\times]0,1[;\, y^{-1}J(y)\intd b\intd y). \]

A $K$-invariant function $f$, $f(k\cdot x)=f(x)$ for $k\in K$ and $x\in X$,
is uniquely determined by its restriction to $A^+\cdot o$.
Abusing notation we write $f(t)=f(\exp(tH)\cdot o)$ if $f$ is $K$-invariant.
Then $f(ka\cdot o)=f(t)$ holds for all $k\in K$, $a=\exp(tH)\in A^+$.
Under the identification \eqref{X-ident-KM-unitinterval} a function is
$K$-invariant if and only if it is a function of the variable $y$ only.

In \cite{MiatelloWallach92resolv} the resolvent of the Laplacian is in effect
constructed as a meromorphic family of convolution operators given by
locally integrable $K$-invariant functions.
The convolution product $u\times v$ of distributions $u,v\in\Dprime(X)$,
one of which is compactly supported, is defined via convolution on the group $G$.
See \cite[Ch.~\RN 2, \S 5.1.]{Helgason84GGA},
where also the following formulae can be found:
\begin{equation}
\label{convolution-facts}
u\times \delta_o=u,\quad \laplaceop(u\times v) =u\times(\laplaceop v),
\end{equation}
$\delta_o$ the Dirac distribution at the origin $o$.
If $u\in C_c(X)$, and if $v\in \Lloc^1(X)$ is $K$-invariant, then
\begin{equation}
\label{eq-convolution-X}
(u\times v)(x)=\int_G u(g\cdot o)v(g^{-1}\cdot x)\intd g
   =\int_X u(g\cdot o)v(g^{-1}\cdot x)\intd gK
\end{equation}
holds for $x\in X$.
Here the $K$-invariance of $v$ is used to prove the last equality.
Conjugation with a representative of the non-trivial Weyl group element $-1$ implies that $a^{-1}\in KaK$ if $a\in A$.
Hence $v(g^{-1}\cdot o)=v(g\cdot o)$, $g\in G$, holds as $v$ is $K$-invariant.
Notice that this implies that the Schwartz kernel of the convolution operator
$C_v:u\mapsto u\times v$ is symmetric in its arguments,
$C_v(x_2,x_1)=C_v(x_1,x_2)= v(x)$, where
$x_j=g_j\cdot o\equiv (b_j,y_j)$ and $x=g_1^{-1}g_2\cdot o\equiv (b,y)$.
It follows from \cite[Lemma 2.1]{BanSchlichtkrull87} that $y^{-1}\leq (y_1y_2)^{-1}$.
Therefore,
\[ \sup_{y_1\geq \eps}\int_{y_2\geq \eps}|C_v(x_1,x_2)|\intd x_2\leq \int_{y\geq\eps^2}|v(x)|\intd x<\infty. \]
By the Schur--Young inequality the left-hand side is a bound for the operator norm on $L^2(X)$
of $1_\eps C_v 1_\eps$,
where $1_\eps$ denotes multiplication by the indicator function of the compact subset $\{y\geq \eps\}\subset X$.
In particular, convolution by $v$ maps $L_c^2(X)$ continuously into $\Lloc^2(X)$.

It is convenient to use the abbreviation
\[ L:=\ksq{\alpha}^{-1}\laplaceop. \]
Then the eigenvalue equation defining $\E_\lambda(X)$ reads $(L-\rho^2+\lambda^2)u=0$.
When restricted to $K$-invariant (radial) functions on $X'$,
$L$ acts as an ordinary differential operator:
\begin{equation}
\label{radial-laplacian-L}
L|_{\operatorname{rad}}
    = -\partial_t^2 -(\partial_t \log J(y)) \partial_t
    = -\theta^2 -(\theta \log J) \theta, \quad \theta= y\partial_y,
\end{equation}
$0<y=e^{-t}<1$.
See \cite[Ch.~\RN 2, Proposition~3.9]{Helgason84GGA}.
For a $K$-invariant function $q(t)$, the eigenvalue equation $(L-\rho^2+\lambda^2)q=0$
in $X'$ is equivalent to the ordinary differential equation 
\begin{equation}
\label{eq-radDelta-u}
\ddot q + (\log J(y))^\cdot \dot q +(\rho^2-\lambda^2)q=0, \quad t>0. 
\end{equation}
A dot denotes one derivative with respect to $t$.
The point $t=0$ is a weakly singular point of \eqref{eq-radDelta-u};
the roots of the characteristic polynomial are $0$ and $1-m_\alpha-m_{2\alpha}\leq 0$.
The spherical function $\sfcn_\lambda$ is the unique solution of 
\eqref{eq-radDelta-u} which is continuous up to $t=0$ and satisfies $\sfcn_\lambda(0)=1$.
There exist functions $h_{j,\lambda}(t)$, $j=1,2$, which are real analytic in a neighbourhood of $t=0$, such that $\sfcn_\lambda(t)$ and
\[ \psi_\lambda(t)=t^{1-m_\alpha-m_{2\alpha}}\big(h_{1,\lambda}(t)+(\log t)h_{2,\lambda}(t)\big)\]
are linearly independent solutions of \eqref{eq-radDelta-u}.
See, for example, \cite[\S 24 \RN 8]{Walter98ODE}.
Observe that $\sfcn_\lambda$ and $\psi_\lambda$ are, as $K$-invariant functions on $X$, locally integrable.

In \cite{MiatelloWallach92resolv} it is shown that,
for every $\lambda\in\C$ such that $2\lambda$ is not a negative integer,
there is a locally integrable $K$-invariant solution $Q_{\lambda}$ of 
\begin{equation}
\label{fundamental-solution-Q}
\big(L-\rho^2+\lambda^2\big) Q_{\lambda} = 2\lambda \cfcn(\lambda)\delta_o.
\end{equation}
In view of \eqref{convolution-facts}, this implies,
if $f\in C_c(X)$ and $\IM\zeta<0$,
\begin{equation}
\label{Rz-as-convolution}
R_\zeta f = \big(2i\langle\zeta,\alpha\rangle \cfcn(i\zeta)\big)^{-1} f\times Q_{i\zeta}.
\end{equation}
For completeness and reference, we summarize in the following lemma the essential
properties of $Q_\lambda$ which were established in \cite{MiatelloWallach92resolv}.
We also give some simplications of the proofs.
\begin{lemma}
\label{lemma-Q-fn}
Suppose $2\lambda$ is not a negative integer.
The equation \eqref{eq-radDelta-u} has a unique solution $Q_\lambda=q$ such that
\begin{equation}
\label{eq-Q-def}
Q_{\lambda}(t) = y^{\rho+\lambda} h_\lambda(y), \quad h_\lambda(0)=1, \quad y=e^{-t},
\end{equation}
and $h_\lambda(y)$ is real analytic in neighbourhood of $y=0$.
For every compact subset $C\subset X$, $\lambda\mapsto Q_\lambda|_C$
maps $\C\setminus -\halb\N$ holomorphically into $L^1(C)$.
Moreover, $Q_\lambda$ solves \eqref{fundamental-solution-Q}.
\end{lemma}
\begin{proof}
Changing variables, $w(y)=q(t)$ with $y=e^{-t}$, \eqref{eq-radDelta-u} becomes
\begin{equation}
\label{eq-radLapl-wy}
\theta^2 w+(\theta J)\theta w+(\rho^2-\lambda^2)w=0, \quad \theta =y\frac{\intd}{\intd y}, \quad 0<y<1.
\end{equation}
The origin $y=0$ is a weakly singular point.
The roots of the characteristic polynomial are $\rho\pm\lambda$.
The assumption implies that the $\rho+\lambda-j$ for $j=1,2,3,\ldots$ are no roots.
A convergence theorem for formal solutions at weakly singular points,
e.g.\ \cite[\S 24 \RN 3]{Walter98ODE}, implies that there exists a function $h$
which is holomorphic in a neighbourhood of $y=0$, $h(0)=1$,
such that $w(y) = y^{\rho+\lambda} h(y)$ is a solution \eqref{eq-radLapl-wy}.
Then $Q_\lambda(t)=w(e^{-t})$ satisfies \eqref{eq-Q-def}.
Since the formal power of $h$ is uniquely determined, the uniqueness of $Q_\lambda$ is clear.
The existence proof in \cite[\S 24]{Walter98ODE} applies the Banach contraction mapping principle,
exhibiting $h(y)$ locally as a uniform limit of holomorphic functions of $y$
and of the parameter $\lambda$.
Therefore, $h_\lambda=h$ and $Q_{\lambda}$ depend holomorphically on the parameter $\lambda$.
Similarly, one shows that $\sfcn_\lambda(t)$ and $h_{j,\lambda}(t)$ depend
holomorphically on $\lambda$.
Expressing $Q_\lambda$ in terms of the fundamental system $(\sfcn_\lambda,\psi_\lambda)$,
we find that, for every $0<T<\infty$, $\lambda\mapsto Q_\lambda|_{[0,T]}$
is a holomorphic map into the weighted $L^1$-space $L^1([0,T];J(y)\intd t)$.

The proof of \eqref{fundamental-solution-Q} given in \cite[Lemma 2.2]{MiatelloWallach92resolv}
uses Green's formula, $K$-invariance, and the key limit formula
\[ \lim_{t\to 0} J(e^{-t})\dot Q_\lambda(t)= -2\lambda\cfcn(\lambda), \]
derived in \cite[Lemma 1.3]{MiatelloWallach92resolv}.
We refer to \cite{MiatelloWallach92resolv} for the detailed argument.
\end{proof}

It follows from the above that, for every $\chi\in\Ccinfty(X)$,
the cut-off resolvent $\chi R_\zeta \chi$ is a holomorphic
map into the space $\LinOp(L^2(X))$ of bounded operators on $L^2(X)$.
Furthermore, it follows from \eqref{eq-convolution-X} and \eqref{Rz-as-convolution}
that the resolvent kernel has the formula
\begin{equation}
\label{defn-Rz}
R_\zeta(g_1\cdot o,g_2\cdot o) = \frac{Q_{i\zeta}(g_1^{-1}g_2\cdot o)}{2i\langle\zeta,\alpha\rangle \cfcn(i\zeta)},
\quad
g_1^{-1}g_2\not\in K.
\end{equation}
Notice the following symmetry properties of the resolvent:
\begin{equation*}
R_\zeta=R_{-\bar\zeta}^*, \quad R'_\zeta = R_\zeta.
\end{equation*}
The first equality refers to the formal adjoint of the resolvent,
and it follows by analytic continuation
because it holds in $\IM\zeta<0$ by the self-adjointness of $\laplaceop$.
The second equality states, in terms the of the transpose (dual operator),
the symmetry of the resolvent kernel in its arguments,
and this holds because $R_\zeta$ is convolution by a $K$-invariant function.

The resolvent kernel admits an asymptotic expansion at the boundary:
\begin{proposition}
Assume that $2i\zeta$ is not a negative integer.
Then there is a real analytic function $r_\zeta$,
defined in an open neighbourhood of $X'\times B\times\{0\}$, such that
\begin{equation}
\label{R-asymptotics}
R_\zeta(x,x') = R_\zeta(x',x) = y^{\rho+i\zeta} r_{\zeta}(x',b,y), \quad
r_{\zeta}(x',b,0)= \frac{e^{(\rho+i\zeta)(A(x',b))}}{2i \langle\zeta,\alpha\rangle \cfcn(i\zeta)}.
\end{equation}
Here $x\equiv (b,y)$ under the identification \eqref{X-ident-KM-unitinterval}.
\end{proposition}
\begin{proof}
The leftmost equality in \eqref{R-asymptotics} restates the symmetry of the resolvent kernel.
Let $g\in G$.
By \eqref{Jacobian-yprime-y} there exists an analytic diffeomorphism $\Phi_g$
between open connected neighbourhoods of $y=0$ such that $\Phi_g(b,y)=(b',y')\equiv g^{-1}\cdot x$
if $x\equiv(b,y)$, $y>0$, and
\begin{equation}
\label{eq-Phi-g-at-B}
\Phi_g(kM,0) = (\kappa(g^{-1}k)M, 0), \quad
y' = y \gamma_g(b,y), \quad \gamma_g(b,0) = e^{\alpha(A(g\cdot o, b))}.
\end{equation}
The function $\gamma_g$ is real analytic and positive.
Let $x'=g\cdot o$ and $x\equiv (kM,y)$.
Then, using the formula~\eqref{defn-Rz} for the resolvent kernel and the asymptotics \eqref{eq-Q-def}, we have
\[
2i \langle\zeta,\alpha\rangle \cfcn(i\zeta) R_\zeta(x',x)
   = Q_{i\zeta}(g^{-1}\cdot x) = (y')^{\rho+i\zeta} h_{i\zeta}(y')).
\]
Now \eqref{R-asymptotics} follows using the second and the third equality of \eqref{eq-Phi-g-at-B}.
\end{proof}

Notice that $Q_{\lambda}$ and $Q_{-\lambda}$ are a fundamental system of
\eqref{eq-radDelta-u} if $2\lambda\not\in\Z$.
Hence the spherical function $\sfcn_{\lambda}=\sfcn_{-\lambda}$ is a linear
combination of $Q_{\lambda}$ and $Q_{-\lambda}$.
Consistent with \eqref{eq-Q-def} and the known spherical function expansion,
\[
\sfcn_{\pm \lambda} = \cfcn(-\lambda) y^{\rho+\lambda} (1+\bigoh(y)) +\cfcn(\lambda) y^{\rho-\lambda} (1+\bigoh(y))
\quad\text{as $y\downarrow 0$,}
\]
the formula
\begin{equation}
\label{eq-spher-fcn-Q-cfcn}
\sfcn_{\pm \lambda} = \cfcn(-\lambda)Q_{\lambda}+\cfcn(\lambda)Q_{-\lambda}
\end{equation}
holds as an equation of functions which are meromorphic in $\lambda$.

\section{Spectral Density}
\label{sect-spectral}

By the spectral theorem there is a unique projection-valued measure $E$,
the spectral measure of the shifted Laplacian:
\[ \laplaceop-\ksq{\rho} = \int_{-\infty}^\infty s \intd E(s). \]

\begin{proposition}
The spectral measure $E$ of $\laplaceop-\ksq{\rho}$ is absolutely continuous, supported in $[0,\infty[$,
and its density is $\intd s$--almost everywhere a difference of resolvent kernels:
\begin{equation}
\label{spec-density-is-resolvent-diff}
\intd E(s) = \big( R_{-\sqrt{s}/\|\alpha\|} -R_{\sqrt{s}/\|\alpha\|}\big) \frac{\intd s}{2\pi i},
\quad s>0.
\end{equation}
\end{proposition}
\begin{proof}
The spectrum of $\laplaceop$ is absolutely continuous,
in particular, $E(F)=0$ holds for every finite subset $F\subset\R$.
By Stone's formula, see e.g.\ \cite[Corollary A.12]{Borthwick07spectral},
\begin{equation}
\label{stone-formula}
E([a,b]) = \lim_{\eps\downarrow 0} \frac{1}{2\pi i}
     \int_a^b \big((\laplaceop-\ksq{\rho}-s -i\eps)^{-1} -(\laplaceop-\ksq{\rho}-s +i\eps)^{-1}\big)\intd s.
\end{equation}
We use the resolvent kernel $R_\zeta$ to evaluate the limit in \eqref{stone-formula}.
The map $\zeta\mapsto s=\ksq{\zeta}$ maps the third and the fourth quadrant of the $\zeta$-plane
to, respectively, the upper and the lower half-plane of the $s$-plane.
Set $a_+=\max(a,0)$.
Formula \eqref{stone-formula} becomes
\[
E([a,b]) = \frac{1}{2\pi i} \int_{a_+}^{b_+} \big( R_{-\zeta} -R_{\zeta}\big) \intd s, 
\quad \zeta=\sqrt{s}/\|\alpha\|>0.
\]
The assertions follow from this.
\end{proof}

\begin{theorem}
\label{thm-resolvent-poisson}
For complex $\zeta$,
\begin{equation}
\label{eq-resol-scatt}
R_{-\zeta}- R_\zeta = \frac{i}{2\langle\zeta,\alpha\rangle \cfcn(i\zeta)\cfcn(-i\zeta)} P_{i\zeta}P_{-i\zeta}^{\prime}.
\end{equation}
Here $P_{i\zeta}$ is viewed as an operator $\Dprime(B)\to\Cinfty(X)$,
and the prime denotes transpose (dual).
\end{theorem}
The formal adjoint and the dual of the Poisson transformation
are related as follows: $P_{i\zeta}^*=P^{\prime}_{-i\bar\zeta}$.
Combining \eqref{eq-resol-scatt} with \eqref{spec-density-is-resolvent-diff},
we obtain a formula for the spectral measure:
\begin{corollary}
In $s>0$ the spectral measure $E$ of $\laplaceop-\ksq{\rho}$ is given by
\begin{equation}
\label{eq-spectral-measure}
\intd E(s) = \frac{1}{4\pi\langle\zeta,\alpha\rangle |\cfcn(i\zeta)|^2} P_{i\zeta}P_{i\zeta}^* \intd s,
\quad \zeta=\sqrt{s}/\|\alpha\|>0.
\end{equation}
\end{corollary}
From \eqref{eq-spectral-measure} we obtain the functional calculus of $\laplaceop$.
\begin{corollary}
Define $\Lambda =\big(\laplaceop-\ksq{\rho}\big)^{1/2}\geq 0$ by the spectral theorem.
For every Borel function $f:\R_+\to\C$,
\[ f(\Lambda) = \int_0^\infty \frac{f(\zeta\|\alpha\|)}{2\pi |\cfcn(i\zeta)|^2} P_{i\zeta}P_{i\zeta}^* \intd \zeta \]
holds.
\end{corollary}

Before giving the proof of Theorem~\ref{thm-resolvent-poisson},
we make some preparations.

\begin{lemma}
\label{lemma-resolvent-diff}
Let $\chi\in\Ccinfty(X)$.
Then
\[ \chi R_{-\zeta} - R_\zeta \chi = R_\zeta [\laplaceop,\chi] R_{-\zeta}. \]
\end{lemma}
\begin{proof}
We have
\begin{align*}
\chi R_{-\zeta} &= R_{\zeta}(\laplaceop -\ksq{\rho} - \ksq{\zeta}) \chi R_{-\zeta} , \\
R_{\zeta} \chi  &= R_{\zeta}\chi (\laplaceop -\ksq{\rho} - \ksq{-\zeta}) R_{-\zeta}.
\end{align*}
Taking the difference the assertion follows.
\end{proof}

\begin{lemma}
Let $\chi\in\Cinfty(X)$ be $K$-invariant. Then
\begin{equation}
\label{eq-commutator-chi-Delta}
[L,\chi] = -2(\theta\chi)(\theta-\rho)-(\theta^2\chi).
\end{equation}
\end{lemma}
\begin{proof}
By assumption, $\chi$ depends only on $y$.
The asserted formula follows by direct computation from the formula~\eqref{radial-laplacian-L} for the radial Laplacian.
\end{proof}

\begin{proof}[Proof of Theorem~\ref{thm-resolvent-poisson}]
The equation \eqref{eq-resol-scatt} to be proved is an equation of meromorphic functions.
It suffices to prove it under the assumption
that $2i\zeta\not\in\Z$ and $\cfcn(i\zeta)\cfcn(-i\zeta)\neq 0$.
Let $\chi_\eps\in\Ccinfty(X)$ such that $\chi_\eps=1$ in $\{y>\eps\}$.
By Lemma~\ref{lemma-resolvent-diff},
\begin{equation}
\label{Resolv-diff-as-limit}
R_{-\zeta}(x_1,x_2) - R_\zeta(x_1,x_2)
  = \ksq{\alpha} \lim_{\eps \to 0} \int_X  R_\zeta(x_1,x) [L, \chi_\eps] R_{-\zeta}(x,x_2) \intd x 
\end{equation}
holds uniformly on compact subsets of $X\times X$.
To evaluate the limit on the right-hand side,
we substitute $x\equiv(b,y)$ in the integrals, using \eqref{eq-measure-dx}.
Moreover, we require that $\chi_\eps(y)=\chi(y/\eps)$ with $\chi\in\Cinfty(\R)$
such that $\chi(y)=0$ if $y<1/2$ and $\chi(y)=1$ if $1\leq y$.

We claim that the last term in the expression \eqref{eq-commutator-chi-Delta}
for the commutator $[L,\chi_\eps]$ does not contribute to the limit in \eqref{Resolv-diff-as-limit}.
To see this, we use \eqref{R-asymptotics} to get
\begin{align*}
\int_X  R_\zeta(x_1,x) & (\theta^2 \chi_\eps)(x) R_{-\zeta}(x,x_2) \intd x \\
 &= \int_B \int_{-1}^1 (\partial_y \theta \chi_\eps)(y) y^{2\rho} J(y) r_\zeta(x_1,b,y)r_{-\zeta}(x_2,b,y)\intd y \intd b.
\end{align*}
Notice that $y^{2\rho}J(y)$ extends smoothly to $y\leq 0$, and equals $1$ at $y=0$.
The limit
\[ \lim_{\eps \to 0} \partial_y y \partial_y \chi_\eps(y) =\partial_y y  \delta(y) = 0 \]
holds in the space of distributions, which implies the claim.

Next we calculate the essential contribution from the commutators $[L,\chi_\eps]$ to the limit.
An application of \eqref{R-asymptotics} gives
\[
(\theta-\rho)R_{-\zeta}(x,x_2)
   = -i\zeta y^{\rho-i\zeta}r_{-\zeta}(x_2,b,y) +  y^{1+\rho-i\zeta}\partial_y r_{-\zeta}(x_2,b,y).
\]
Therefore
\begin{align*}
\int_X  R_\zeta(x_1,x) & (\theta \chi_\eps)(x)(\theta-\rho) R_{-\zeta}(x,x_2) \intd x \\
 &= -i\zeta \int_B \int_{-1}^1 (\partial_y \chi_\eps)(y)
       y^{2\rho} J(y) \big(r_\zeta(x_1,b,y)r_{-\zeta}(x_2,b,y)+\bigoh(y)\big)\intd y \intd b.
\end{align*}
So finally, we evaluate the limit in \eqref{Resolv-diff-as-limit} as follows:
\[ R_{-\zeta}(x_1,x_2) - R_\zeta(x_1,x_2) = \ksq{\alpha} 2i\zeta \int_B r_\zeta(x_1,b,0) r_{-\zeta}(x_2,b,0) \intd b.  \]
In view of \eqref{def-Poisson-transform} and of the last part of \eqref{R-asymptotics}, the theorem follows.
\end{proof}

\section{Resonances}
\label{sect-Resonances}

Recall that the poles of the resolvent $R_\zeta$ are called resonances.
The resonances and their residues were determined in \cite[Theorem 3.1]{MiatelloWill99resid}.
See also \cite[Theorem 3.8]{HilgertPasquale09reson}.
The set of resonances is either empty or forms an arithmetic progression
along the positive imaginary axis starting at $i\rho$.
We partially rederive these results below.

Resonances will be seen to correspond to zeros of the function $\cfcn(i\zeta)\cfcn(-i\zeta)$.
The meromorphic extension of the $\cfcn$-function is given by the formula
\begin{equation*}
\cfcn(\lambda)=c_0\frac{\Gamma(\lambda)2^{-\lambda}}%
  {\Gamma(\frac{1}{2}(\frac{m_{\alpha}}{2}+1+\lambda))\Gamma(\frac{1}{2}(\frac{m_{\alpha}}{2}+m_{2\alpha}+\lambda))},
\end{equation*}
where $c_0$ is determined by $\cfcn(\rho)=1$.

\begin{remark} 
\label{remark-cfcn-zeros}
We recall the location of the zeros of $\cfcn(i\zeta)\cfcn(-i\zeta)$ from \cite[Lemma 2.1]{HilgertPasquale09reson}:
When $m_{2\alpha}=0$ and $m_\alpha$ is even, 
the function $\cfcn(i\zeta)\cfcn(-i\zeta)$ has no zeros.
Otherwise, the zeros are simple and are located at the points $\zeta=i(\rho+jk)$, $k\in \N_0$.
Here $j=2$ if $m_{2\alpha}\neq 0$,  and $j=1$ if $m_{2\alpha}=0$ and $m_{\alpha}$ is odd.
\end{remark}

Next, we determine the resonances and the residues.

\begin{proposition}
The resonances $\zeta$  are precisely the zeros of $\cfcn(i\zeta)\cfcn(-i\zeta)$ with $\IM \zeta>0$.
Each resonance $\zeta$ is a simple pole with residue
\begin{equation}
\label{eq-residue-of-resonance}
\Res R_\zeta 
 = \frac{-1}{2\langle\zeta,\alpha\rangle \cfcn'(i\zeta)\cfcn(-i\zeta)} P_{i\zeta}P_{-i\zeta}^{\prime}.
\end{equation}
Here $\cfcn'$ is the derivative of the $\cfcn$-function.
The rank of $\Res R_\zeta$ is finite.
Alternatively,
\begin{equation*}
\Res R_\zeta f
 = \frac{-1}{2\langle\zeta,\alpha\rangle \cfcn'(i\zeta)\cfcn(-i\zeta)} f\times \sfcn_{i\zeta}
\end{equation*}
holds for $f\in \Eprime(X)$.
\end{proposition}
The location of the resonances differs from that given in \cite[Theorem 3.8]{HilgertPasquale09reson}
by a scaling factor; this is because in the cited paper $z\mapsto R_{-z/\|\alpha\|}$ is taken as the resolvent.
\begin{proof}
Let $\zeta$ be a resonance.
It follows from the holomorphicity assertion in Lemma \ref{lemma-Q-fn} and \eqref{defn-Rz}, that $\IM \zeta > 0$.
Formula \eqref{eq-resol-scatt} then implies that  $\cfcn(i\zeta)=0$ and \eqref{eq-residue-of-resonance}.
Conversely, if $\cfcn(i\zeta)\cfcn(-i\zeta)=0$ then the residue is non-zero because
$P_{\lambda}P_{-\lambda}^{\prime}\neq 0$ if $\lambda$ is real.
The latter follows because the Schwartz kernel of the Poisson
transform $P_\lambda$ is a positive function when $\lambda$ is real.
Since, by Remark \ref{remark-cfcn-zeros}, the zeros of $\cfcn(i\zeta)\cfcn(-i\zeta)$ are simple, so are the poles of $R_\zeta$.
The convolution formula for the residue follows from
\eqref{Rz-as-convolution} and \eqref{eq-spher-fcn-Q-cfcn}.
The rank-finiteness of residues
is shown in \cite[Theorem 3.1]{MiatelloWill99resid} and  \cite[Corollary 3.7]{HilgertPasquale09reson}.
\end{proof}

The rank of the resolvent residue operator $\Res R_\zeta$ is called the multiplicity of the resonance $\zeta$.
As discussed in \cite[Theorem 4.1]{MiatelloWill99resid} and  \cite[Section 3]{HilgertPasquale09reson},
the rank can be explicitly determined using the Weyl dimension formula.
For explicit formulae for the multiplicities, we refer to \cite[Remark 4.2]{MiatelloWill99resid}.

\section{Boundary Values}
\label{sect-Boundary}

Using \eqref{X-ident-KM-unitinterval},
we identify continuous functions and differential operators on $X$ with
their pullbacks to $B\times]0,1[$.
The Laplacian $L=\ksq{\alpha}^{-1}\laplaceop$
uniquely extends to a differential operator with real analytic coefficients on $B\times]-1,1[$:
\begin{equation}
\label{laplacian-in-ky-coord}
L = -\theta^2 +2\rho \theta + yC(b,y,\partial_b,\theta), \quad \theta=y\partial_y.
\end{equation}
See \cite[(F) on page~13]{Kashiwara78eigen}.
The space $\E_{\lambda}(X)$ is the solution space of the eigenequation $(L-\rho^2+\lambda^2)u=0$.
The operator $L$ is of Fuchsian type (of weight zero) in the sense of \cite{BaouendiGoulaouic73};  
in \cite{KashiwaraOshima77} and \cite{Oshima83bdryval}, $L$ is called regular singular in the weak sense.
The polynomial
\[ I(s) = y^{-s}L y^s|_{y=0+} = -s^2 +2\rho s \]
is the indicial polynomial of $L$.

Solutions of Fuchsian type equations can be assigned boundary values.
More specifically,
boundary values of $u\in\Et_{\lambda}(X)$ are defined with respect to characteristic exponents,
that is the roots $s=\rho\pm \lambda$ of the indicial equation $I(s)-\rho^2+\lambda^2=0$.
To define its boundary values, $u$ needs to satisfy the eigenequation only in a neighbourhood of infinity.
So let $X_\eta$, $0<\eta\leq 1$, denote the open subset of $X$ defined by $y<\eta$, and denote by
\[
\Et_\lambda(X_\eta) =\{u\in\cup_{m}\Sob{-m}(B\times]0,\eta[)\mid (L-\rho^2+\lambda^2)u=0\}
\]
the (DFS)-space of tempered eigenfunctions with spectral parameter $\lambda$ in $X_\eta$.
\begin{lemma}
\label{lemma-bv-def-and-limit}
For every $0<\eta\leq 1$, and every $\lambda\in\C$, $-2\lambda\not\in \N_0$,
there is a continuous linear operator $\bv_{\rho-\lambda}:\Et_{\lambda}(X_\eta)\to\Dprime(B)$ such that
\begin{equation}
\label{eq-bv-is-limit}
\bv_{\rho-\lambda} u = \lim\nolimits _{y\to 0+} y^{-\rho+\lambda} u
\quad\text{in $C(B)$}
\end{equation}
holds, provided the uniform limit on the right exists.
If $f$ is a holomorphic map from a bounded open set $\Lambda\subset\C\setminus -\halb\N_0$
into the space $\LinOp(V,\Sob{-m}(B\times]0,\eta[))$ of bounded linear operators
from a Banach space $V$ into the Sobolev space of some order $-m$,
such that the range of $f(\lambda)$ is contained in $\E_\lambda(X_\eta)$,
then $\Lambda\ni\lambda\mapsto \bv_{\rho-\lambda} f(\lambda)$ is a holomorphic map
into $\LinOp(V,\Sob{-m'}(B))$ for some $m'=m'(m,\Lambda)$.
If $0<\eta'<\eta$ and if $\bv_{\rho-\lambda}'$ denotes the
corresponding operator from $\Et_{\lambda}(X_{\eta'})$ into $\Dprime(B)$,
then $\bv_{\rho-\lambda}u= \bv_{\rho-\lambda}'(u|_{X_{\eta'}})$ holds for $u\in\Et_{\lambda}(X_\eta)$.
For $\eta=1$, $X_\eta=X$ holds, and $\bv_{\rho-\lambda}:\Et_\lambda(X)\to\Dprime(B)$ is $G$-equivariant.
\end{lemma}
Lemma~\ref{lemma-bv-def-and-limit} is proved in Appendix~\ref{appx-bvmap}.
The definition of the boundary value map is the same as the one given in \cite[Definition 3.3]{Oshima83bdryval},
but the limit formula \eqref{eq-bv-is-limit}, which is needed in this generality below, seems to be new.
See \cite[Lemma~9.5]{BanSchlichtkrull87} for a similar result.

If $f\in\Eprime(X)$, then, by \eqref{fundamental-solution-Q} and \eqref{Rz-as-convolution}, in $X\setminus\supp(f)$, $u=R_\zeta f$ solves $(L-\rho^2-\zeta^2)u=0$.
It follows from \eqref{R-asymptotics} that $u$ is tempered.
Hence the boundary values of $u$ are defined.
\begin{corollary}
\label{cor-poisson-is-by-of-resolvent}
The transpose of the Poisson transform $P_{\lambda}:\Dprime(B)\to\Cinfty(X)$ equals
\[ P_{\lambda}^{\prime}= 2\langle\lambda,\alpha\rangle \cfcn(\lambda) \bv_{\rho+\lambda} R_{-i\lambda}:\Eprime(X)\to \Cinfty(B). \]
\end{corollary}
\begin{proof}
This follows from \eqref{def-Poisson-transform}, \eqref{R-asymptotics}, and \eqref{eq-bv-is-limit}.
\end{proof}

\begin{remark}
\label{remark-bv-definitions}
The boundary value $\bv_{\rho-\lambda}u$ of $u$ is encoded in $\tilde u$,
by which is meant the unique extension of $u$ to $B\times]-1,\eta[$,
which solves the eigenvalue equation $(L-\rho^2+\lambda^2)\tilde u=0$, and which is supported in $y\geq 0$.
The construction of $\tilde u$ is done as in Appendix~\ref{appx-bvmap}.
The original definition of boundary values, \cite[Definition 4.8]{KashiwaraOshima77},
assumes the stronger condition $2\lambda\not\in\Z$ (condition (A) in the cited paper).
It relies on the following microlocal equation which determines the boundary values uniquely:
\begin{equation}
\label{eq-KO-bv-def}
\tilde u \equiv A_- (y^{\rho-\lambda} \otimes \bv_{\rho-\lambda}u) + A_+ (y^{\rho+\lambda} \otimes \bv_{\rho+\lambda}u)
\quad\text{at $N^*B\setminus\{0\}$.}
\end{equation}
Here $A_\pm$ are zeroth order (analytic) pseudodifferential operators
with principal symbols equal to one at $N^*Y\setminus\{0\}$.
Recall that two distributions are equal as microfunctions at a conic subset $\Gamma\subset T^*X\setminus\{0\}$
if the analytic wavefront set of their difference has empty intersection with $\Gamma$.
If, with $v_\pm$ real analytic at $y=0$,
$u = y^{\rho-\lambda} v_- + y^{\rho+\lambda} v_+$
holds, then $\bv_{\rho\pm \lambda}u={v_\pm}|_{y=0}$.
See \cite[Theorem 3.5]{Oshima83bdryval} and \cite[Theorem 5.14]{KashiwaraOshima77}.
\end{remark}

It follows from \eqref{eq-KO-bv-def} and the Holmgren--Kashiwara Theorem
that $u$ vanishes when both boundary values vanish, \cite[Proposition 4.9]{KashiwaraOshima77}.
The following result implies that, for global eigenfunctions,
the conclusion already holds when $\cfcn(\lambda)\neq 0$ and when
a single boundary value vanishes.

\begin{theorem}[{\cite{Lewis78eigenf}, \cite{Kashiwara78eigen}, \cite[Corollary~5.5]{OshimaSekiguchi80}}]
\label{Thm-KKMOOT}
Let $\lambda\in\C\setminus -\halb\N_0$.
Then $\bv_{\rho-\lambda}$ is a $G$-equivariant map from $\Et_{\lambda}(X)$ into $\Dprime(B)$, and
\begin{equation}
\label{eq-beta-Poisson-equal-id}
\bv_{\rho-\lambda} P_{\lambda} =\cfcn(\lambda)\Id.
\end{equation}
If, in addition, $\cfcn(\lambda)\neq 0$, then $\bv_{\rho-\lambda}$
is injective, hence $P_{\lambda}$ an isomorphism.
\end{theorem}

Notice that the assumption on $\lambda$ rules out the poles of $\cfcn(\lambda)$.
Observe that the boundary values of the spherical function are the constants
\begin{equation}
\label{eqbv-of-spher-fcn}
\bv_{\rho\pm \lambda}\sfcn_{\lambda} = \cfcn(\mp \lambda).
\end{equation}

\begin{proof}
Let $f\in C(B)$.
Set $u=P_{\lambda}f$.
Define  the function $v:B\times[0,1[\to\C$
by $v=y^{-\rho+\lambda}u$ when $y>0$ and by $v=\cfcn(\lambda)f$ at $y=0$.
By \eqref{Helgason-Fatou-thm} we have $v\in C(B\times[0,1])$ when $\RE\lambda>0$.
It follows from Lemma~\ref{lemma-bv-def-and-limit} that the
equation $\bv_{\rho-\lambda}P_{\lambda}f=\cfcn(\lambda)f$ holds.
Using the holomorphicity assertion of Lemma~\ref{lemma-bv-def-and-limit},
the equation follows when $-2\lambda\not\in \N_0$ by analytic continuation.
Since $C(B)$ is dense in $\Dprime(B)$, \eqref{eq-beta-Poisson-equal-id} follows.

To prove the last assertion, we proceed as in the proof of the
Theorem in \cite[\S 5]{Kashiwara78eigen}.
Suppose there exists $0\neq u\in\Et_\lambda(X)$, $\bv_{\rho-\lambda}u=0$.
By the $G$-equivariance of the boundary value map we can assume that $u(o)=1$.
Then $\sfcn_{\lambda}=\int_K T_\lambda(k) u\intd k$ is the spherical function.
By the $G$-equivariance of the boundary value map we have
\[
\bv_{\rho-\lambda}\sfcn_\lambda =
\int_K \bv_{\rho-\lambda} T_\lambda(k) u\intd k = 
\int_K \pi_\lambda(k) \bv_{\rho-\lambda}  u\intd k = 0.
\]
Now \eqref{eqbv-of-spher-fcn} implies $\cfcn(\lambda)=0$.
\end{proof}

\section{Scattering Poles}
\label{sect-Scatt}

Following the ideas of geometric scattering theory, \cite{Melrose95geoscatt},
and consistent with Remark~\ref{remark-bv-definitions},
we require that the scattering matrix $S_\zeta$ satisfies
\begin{equation}
\label{def-scatt-geometric}
S_\zeta: \bv_{\rho-i\zeta}u\mapsto \bv_{\rho+i\zeta}u
\quad\text{if $(L-\rho^2-\zeta^2)u=0$.}
\end{equation}
For $\zeta\in\C$ such that $\cfcn(i\zeta)\neq 0$ and $2i\zeta\not\in\N_0$,
we define the scattering matrix as follows:
\begin{equation}
\label{def-scatt-op}
S_\zeta = \cfcn(i\zeta)^{-1} \bv_{\rho+i\zeta}P_{i\zeta}.
\end{equation}
Then, by Theorem~\ref{Thm-KKMOOT}, \eqref{def-scatt-geometric} holds for $u$ belonging to the range of $P_{i\zeta}$.
It follows from Lemma~\ref{lemma-bv-def-and-limit} that, for every $f\in\Dprime(B)$,
$S_\zeta f$ is holomorphic in $\zeta\in\C$ except for at most simple poles in $\IM\zeta>0$
and singularities in $(2i)^{-1}\N_0$.
Theorem~\ref{Thm-KKMOOT} implies also $P_{-i\zeta}\bv_{\rho+i\zeta}=\cfcn(-i\zeta)\Id$ if $\cfcn(-i\zeta)\neq 0$.
Therefore,
\begin{equation}
\label{def-scatt-op-intertwines}
\cfcn(i\zeta) P_{-i\zeta}S_\zeta = \cfcn(-i\zeta) P_{i\zeta}.
\end{equation}
Hence $\cfcn(i\zeta)P_{-i\zeta}S_\zeta S_{-\zeta} = \cfcn(-i\zeta) P_{i\zeta} S_{-\zeta}= \cfcn(i\zeta) P_{-i\zeta}$,
and
\begin{equation*}
S_\zeta S_{-\zeta} = \Id.
\end{equation*}

Corollary~\ref{cor-poisson-is-by-of-resolvent} combined with \eqref{def-scatt-op}
gives a formula which expresses the scattering matrix as boundary values of the resolvent:
\begin{equation*}
S_\zeta = 2i\langle\zeta,\alpha\rangle \bv_{\rho+i\zeta} (\bv_{\rho+i\zeta}R_\zeta)^{\prime}.
\end{equation*}
On the level of Schwartz kernels this amounts to taking boundary values in both variables.
Moreover, \eqref{eq-resol-scatt} and \eqref{def-scatt-op-intertwines}
imply the following relation between the resolvent, the scattering matrix,
the Poisson transform, and the $\cfcn$-function:
\begin{equation*}
R_{\zeta}- R_{-\zeta} = \frac{1}{2i\langle\zeta,\alpha\rangle \cfcn(-i\zeta)^2} P_{-i\zeta}S_{\zeta}^{\prime}P_{-i\zeta}^{\prime}.
\end{equation*}

The significance of the scattering matrix in representation theory is that it
intertwines spherical principal series representations.
The standard intertwiner $\T_\lambda$ from $\pi_\lambda$ to $\pi_{-\lambda}$
satisfies $\pi_{-\lambda}(g) \T_{\lambda} = \T_{\lambda} \pi_\lambda(g)$,
and it is, in the induced picture, determined by the integrals 
\[ \T_\lambda f(g)=\int_{\overline{N}}f(gw^*\bar n)d\bar n. \]
Here $w^*$ is a representative in the normalizer of $\a$ in $K$,
of the non-trivial Weyl group element $w=-1$.
See \cite[\S 6]{Kashiwara78eigen} and \cite[Ch.~\RN 7 \S 7]{Knapp86repexa}.
Specialized to our rank one situation, \cite[Proposition 6.1]{Kashiwara78eigen} implies
\begin{equation}
\label{eq-bv-P-T}
\bv_{\rho+\lambda}P_\lambda =\T_\lambda.
\end{equation}
According to \cite[Theorem 7.12]{Knapp86repexa},
$\T_\lambda f$, initially defined as an absolutely convergent integral when $\RE\lambda<0$,
extends meromorphically to the $\lambda$-plane with at most simple poles in the non-negative half-integers.
From the definition \eqref{def-scatt-op} and \eqref{eq-bv-P-T}
we get
\begin{equation}
\label{eq-Scatt-Intertwiner}
S_\zeta = \cfcn(\lambda)^{-1} \T_{\lambda}, \quad \lambda=i\zeta.
\end{equation}
This shows that the scattering matrix depends meromorphically on $\zeta$.
We have the following result on scattering poles, that is the possible
poles of the $S_\zeta f$, $f\in\Cinfty(B)$.
\begin{theorem}
\label{thm-scatt-poles-are-resonances}
The scattering poles are at most simple and located on the imaginary axis minus the origin.
The scattering poles in $\IM \zeta>0$ are precisely the resonances.
At a resonance $\zeta$, $\bv_{\rho+i\zeta}$ maps the range of
$\Res R_\zeta$ isomorphically onto the range of $\Res S_\zeta$;
in particular, the residues of the resolvent
and of the scattering matrix have the same finite rank.
The non-resonance scattering poles are the poles of the standard intertwiner $\T_{i\zeta}$;
they are contained in $(2i)^{-1}\N$.
At $\IM\zeta<0$, the residue $\Res S_\zeta$ is a non-zero multiple of the residue $\Res \T_{i\zeta}$.
\end{theorem}
\begin{proof}
From \eqref{eq-Scatt-Intertwiner} and the meromorphy of $\cfcn(\lambda)$ and $\T_\lambda$
it is clear that $S_\zeta$ is meromorphic with at most simple poles.
Furthermore, it follows that the scattering poles are either resonances
or poles of the standard intertwiner $\T_{i\zeta}$.
The simple pole of the $\cfcn$-function at the origin cancels a possible pole of the intertwiner.
Hence zero is not among the scattering poles.

Let $\zeta$ be a resonance, $\cfcn(i\zeta)=0$ and $\zeta\in i\R_+$.
From the definition \eqref{def-scatt-op} we obtain the following formula for residues:
\[ \Res S_\zeta = (i\cfcn'(i\zeta))^{-1} \bv_{\rho+i\zeta}P_{i\zeta}.  \]
In view of \eqref{eq-residue-of-resonance} this gives
\begin{equation}
\label{eq-Scatt-vs-Resolv-Residue}
(\Res S_\zeta) P_{-i\zeta}^{\prime}
   = 2i\langle\zeta,\alpha\rangle \cfcn(-i\zeta) \bv_{\rho+i\zeta} (\Res R_\zeta).
\end{equation}
It follows from Theorem~\ref{Thm-KKMOOT} that $\bv_{\rho+i\zeta}$ is injective because $\cfcn(-i\zeta)\neq 0$.
Moreover, $P_{-i\zeta}$ is an isomorphism onto $\Et_{i\zeta}(X)$;
thus $P_{-i\zeta}'$ has dense range in $\Cinfty(B)$.
Now it follows from \eqref{eq-Scatt-vs-Resolv-Residue} that $\bv_{\rho+i\zeta}$ maps
the range of $\Res R_\zeta$, which is finite-dimensional,
bijectively onto the range of $\Res S_\zeta$.

Let $\zeta$ be a scattering pole in $\IM\zeta\leq 0$.
Then $\zeta\neq 0$, $2i\zeta\in\N$, and $\zeta$ is a pole of $\T_{i\zeta}$.
Using \eqref{eq-Scatt-Intertwiner} and $\cfcn(i\zeta)\neq 0$,
we obtain the formula $\Res S_\zeta=\cfcn(i\zeta)^{-1}\Res \T_{i\zeta}$ for the residues.
\end{proof}

\appendix
\section{Proof of Lemma~\ref{lemma-bv-def-and-limit}}
\label{appx-bvmap}

We consider the Riemannian manifold
$M=B\times]-\eta,\eta[$ and the open subset $\Omega=B\times]0,\eta[$.
Let $r:\Dprime(M)\to\Dprime(\Omega)$ denote the restriction map,
and $\Delta$ the nonnegative Laplacian on $M$.
By the variational theory of the Dirichlet problem,
$\Delta^m+1$ is a strongly elliptic differential operator of order $2m$
which maps $\SobDot{m}(\Omega)$ isomorphically onto $\Sob{-m}(\Omega)$.
It is standard, \cite[Lemma 12.2]{LionsMagenes72bvp}, that there exist
extension operators $e_m:\Sob{m}(\Omega)\to\Sob{m}(M)$, $m\in\Z$, that is, $re_m$ is the identity map.
Allowing for a loss of regularity, extensions can be chosen with supports contained
in the closure $\bar\Omega$ of $\Omega$.
In fact, for $m=0,1,2,\ldots$ there exist extension operators
\begin{equation*}
E_m:\Sob{-m}(\Omega)\to\Sob{-2m}(M), \quad rE_mu=u,\quad \supp(E_mu)\subset\bar\Omega,
\end{equation*}
$u\in\Sob{-m}(\Omega)$.
The definition of $E_0$ is clear.
If $m>0$ we define $E_m=(\Delta^m+1)E_0(\Delta^m+1)^{-1}$.

Next we consider multiplication of elements in $\Sob{-m}(\Omega)$ by $y^s$, $s\in\C$.
Recall from Section~\ref{sect-Poisson} the definition of $\SobDot{m}(\Omega)$,
and the continuous inclusion $\SobDot{m}(\Omega)\subset C_b^k(M)$ which
holds for $m> k+\dim(M)/2$ by the Sobolev Lemma.
If $\RE s>-N$ and $k\geq N$, then $u\in C_b^k(M)$, $\supp(u)\subset\bar\Omega$,
implies that $y^su$, defined to be zero when $y\leq 0$, belongs to $C_b^{k-N}(M)$.
Hence we have $y^s\cdot \SobDot{m}(\Omega)\subset \SobDot{k-N}(\Omega)$.
It follows from the closed graph theorem that the multiplication map $y^s\cdot$ is continuous.
Hence, for $m_j,N\in\N$, it follows by duality that
\begin{equation*}
y^s\cdot:\Sob{-m_1}(\Omega)\to\Sob{-m_2}(\Omega) \quad \text{if $\RE s>-N$, and $m_2>m_1+N+\halb\dim(M)$.}
\end{equation*}
Furthermore, the map $s\mapsto y^s\cdot$ is holomorphic with values
in the Banach space of bounded linear operators
from $\Sob{-m_1}(\Omega)$ to $\Sob{-m_2}(\Omega)$.
The complex derivative at $s$ is the multiplication operator $\log(y)y^s\cdot$.

Let $\psi\in\Ccinfty(]-\eta,\eta[)$.
Then we define the push-forward $\iota_\psi:\Sob{-m}(M)\to\Sob{-m}(B)$ by integration:
$\iota_\psi(u)(b)=\int u(b,y)\psi(y)\intd y$.
Observe that $\iota_\psi(w\otimes\delta(y))=\psi(0)w$.

Given $\lambda\in\C$, we consider the differential operator
\[
L_\lambda = y^{-\rho+\lambda} (L-\rho^2+\lambda^2)y^{\rho-\lambda} 
    = -\theta (\theta-2\lambda) + y C(b,y,\partial_b,\rho-\lambda+\theta).
\]
As before $L=\ksq{\alpha}^{-1}\laplaceop$ is the Laplacian.
We have set $\theta=y\partial_y$, and we used \eqref{laplacian-in-ky-coord}.
The indicial polynomial of $L_\lambda$ is $I_\lambda(s)=-s(s-2\lambda)$.
Because of $I_\lambda(0)=0$, $y^{-1}L_\lambda$ is a
differential operator with real analytic coefficients in $M$.
The significance of the indicial polynomial is seen in the formula
\begin{equation}
\label{eq-Lz-on-tensor}
L_\lambda \big(w\otimes \delta^{(j)}(y)\big)= I_\lambda(-j-1) w\otimes \delta^{(j)}(y)
    + \sum\nolimits _{\ell<j} (L_{\lambda,j,\ell}w)\otimes \delta^{(\ell)}(y),
\end{equation}
$w\in\Dprime(B)$.
Here the $L_{\lambda,j,\ell}$'s are differential operators on $B$.

We consider the action of $L_\lambda$ on distributions supported in $y=0$.
Denote by $\Sob{-m}_B(M)$ the subspace of $\Sob{-m}(M)$ which
consists of elements having their supports contained in $B\times\{0\}$.
If $f\in\Sob{-m}_B(M)$ and $m'>m+\dim(B)/2$, then $f$
is a finite sum of tensor products
\[ f(b,y) = \sum\nolimits _{j\leq J} f_j(b)\otimes\delta^{(j)}(y), \quad f_j\in \Sob{-m'}(B); \]
see \cite[Example~5.1.2]{Hormander90anaOne}, and invoke the Sobolev Lemma.
Assume that $2\lambda$ is not a negative integer.
Then the equation $L_\lambda w=f$ has a unique solution $w=S_\lambda f$ with $\supp w\subset \{y=0\}$.
In fact, using \eqref{eq-Lz-on-tensor} and the assumption on $\lambda$,
we inductively determine $w_J,\dots,w_1,w_0$
such that $w(b,y) = \sum\nolimits _{j\leq J} w_j(b)\otimes\delta^{(j)}(y)$
solves the equation.
Moreover, if the $f_j$'s depend holomorphically on the parameter $\lambda$, then so do the $w_j$'s.
For some $n$ depending on $m$,
$S_\lambda:\Sob{-m}_B(M)\to\Sob{-n}_B(M)$ is a bounded linear operator.

Let $u\in\Sob{-m_1}(\Omega)$.
Let $\lambda\in\C$, $2\lambda\not\in -\N_0$.
Assume $(L-\rho^2+\lambda^2)u=0$.
Set $v=y^{-\rho+\lambda}u$.
Then $L_\lambda v=0$.
For some $m$, $v\in\Sob{-m}(\Omega)$ holds.
Thus $f=L_\lambda E_m v$ is defined, and $\supp(f)\subset\{y=0\}$.
Define $\tilde v=E_mv-S_\lambda f\in\Dprime(M)$.
Then $\tilde v$ is the unique extension of $v$
which satisfies $L_\lambda \tilde v=0$ in $M$ and $\supp(\tilde v)\subset\{y\geq 0\}$.
It follows that $(y^{-1}L_\lambda)\tilde v= w\otimes\delta(y)$.
The boundary value of $u$ is, by definition, $\bv_{\rho-\lambda}u=(2\lambda)^{-1}w$.
Observe that the uniqueness of the extension $\tilde v$
implies the last assertion of Lemma~\ref{lemma-bv-def-and-limit}.

Next we prove the formula \eqref{eq-bv-is-limit}.
\begin{lemma}
\label{lemma-bv-def-justified}
Assume that $u=vy^{\rho-\lambda}$ with $v\in C(\bar\Omega)$.
Then $\bv_{\rho-\lambda}u=v|_{y=0}$.
\end{lemma}
\begin{proof}
To prove the lemma it suffices to show that
\begin{equation}
\label{eq-proof-bv}
(y^{-1}L_\lambda) E_0 v = 2\lambda v|_{y=0}\otimes\delta(y).
\end{equation}
Indeed, this will in particular imply that $\tilde v=E_0 v$ satisfies $L_\lambda\tilde v=0$.

Fix $\chi\in\Cinfty(\R)$, $0\leq \chi\leq 1$, $\chi(y)=0$ if $y<1/2$, and $\chi(y)=1$ if $y\geq 1$.
For $0<\eps<1$ set $\chi_\eps(y)=\chi(y/\eps)$.
Clearly, $\chi_\eps v\to E_0 v$ in $\Lloc^2(M)$ as $\eps\to 0$.
We calculate the limit as $\eps\to 0$ of
\[ (y^{-1}L_\lambda)(\chi_\eps v) = [y^{-1}L_\lambda,\chi_\eps] v. \]
Here $\chi_\eps$ is viewed as a multiplication operator, and the bracket denotes a commutator.
Using the formula
$y^{-1}L_\lambda = (2\lambda-\theta-1)\partial_y + C(b,y,\partial_b,\rho-\lambda+\theta)$
we can write:
\[
[y^{-1}L_\lambda,\chi_\eps]
  = (2\lambda-\theta-1)[\partial_y,\chi_\eps] + \sum\nolimits _{k\leq 2} Q_k(b,y,\partial_b,\partial_y) \ad(\theta)^k \chi_\eps,
\]
where $\ad(\theta)=[\theta,\cdot]$.
Observe that, $\ad(\theta)^k\chi_\eps$ is supported in $0<y\leq \eps$ and uniformly bounded.
Therefore, $\ad(\theta)^k\chi_\eps v\to 0$ as $\eps\to 0$ in $\Dprime(M)$.
Moreover, we have $\lim_{\eps\to 0} [\partial_y,\chi_\eps]=\delta(y)$, and $(\theta+1)\delta(y)=0$.
This implies \eqref{eq-proof-bv}, hence the lemma.
\end{proof}

The boundary value map is the restriction to the eigenspaces $\Et_{\lambda}(X_\eta)$
of a holomorphic operator-valued map which maps $\lambda$ to a bounded linear operator
\begin{equation}
\label{eq-bv-on-Sobolev}
\bv_{\rho-\lambda}:\Sob{-m_1}(\Omega)\to \Sob{-m_2}(B), \quad \RE\lambda>-N.
\end{equation}
In fact, summarizing the arguments preceeding Lemma~\ref{lemma-bv-def-justified}
we define the boundary value map as a composition of operators:
\begin{equation*}
\bv_{\rho-\lambda}u = \frac{1}{2\lambda} \iota_\psi (y^{-1}L_\lambda)
    \big(E_{m} + S_\lambda (E_{m+2} L_\lambda-L_\lambda E_{m})\big) y^{-\rho+\lambda}\cdot u.
\end{equation*}
Here $\psi(0)=1$, and $m, m_2$ depend on $m_1, N$.
The assertion about holomorphy in Lemma~\ref{lemma-bv-def-and-limit} now follows from the fact that
$\lambda\mapsto \bv_{\rho-\lambda} f(\lambda)$ is a product of holomorphic mappings.
Notice that the map \eqref{eq-bv-on-Sobolev} is not canonical, whereas
\[ \bv_{\rho-\lambda}:\Et_\lambda(X_\eta)\to\Dprime(B), \]
the restriction to the eigenspace, is.

Finally, we prove the $G$-equivariance of the boundary value map.
See \cite[Theorem~3.7]{Oshima83bdryval} for coordinate invariance.
Let $g\in G$.
Recall the map $\Phi_g:(b,y)\mapsto(b',y')$ in \eqref{eq-Phi-g-at-B} which extends $x\mapsto g^{-1}\cdot x$
to a real analytic diffeomorphism between connected open neighbourhoods of $y=0$.
Denoting by $\Phi_g^*$ the pullback operator, $\Phi_g^* u(b,y)=u(b',y')$,
we have $\Phi_g^*\circ L=L\circ\Phi_g^*$.
Observe that $\Phi_g^*=T_\lambda(g)$ when suitably restricting to $y>0$.
Using \eqref{eq-Phi-g-at-B}, we find that
\begin{equation}
\label{eq-transf-Llambda}
\gamma_g^{\rho-\lambda} \Phi_g^* L_\lambda =L_\lambda\gamma_g^{\rho-\lambda} \Phi_g^*.
\end{equation}
Let $u\in\Et_\lambda(X)$.
Denote by $\tilde u$ the unique distribution in $M$ which satisfies
\[
\tilde u|_{y>0}=y^{\lambda-\rho}u, \quad \supp(\tilde u)\subset\{y\geq 0\}, \quad L_\lambda\tilde u=0.
\]
Then, by definition,
$(y^{-1}L_\lambda) \tilde u = 2\lambda \bv_{\rho-\lambda}u\otimes\delta(y)$.
Set $v=T_\lambda(g)u$, and $\tilde v= \gamma_g^{\rho-\lambda}\Phi_g^*\tilde u$.
Since $\Phi_g$ preserves $y\geq 0$, $\supp(\tilde v)\subset\{y\geq 0\}$ holds.
In view of \eqref{eq-transf-Llambda}, $L_\lambda \tilde v=0$ holds.
Moreover, using \eqref{eq-Phi-g-at-B} we get
\[ \tilde v(b,y)= \gamma_g(b,y)^{\rho-\lambda} (y')^{\lambda-\rho} u(b',y') =y^{\lambda-\rho}v(b,y) \quad\text{in $y>0$.} \]
Therefore, we shall calculate the boundary value of $v$ from
\[
2\lambda \bv_{\rho-\lambda}v\otimes\delta(y) = (y^{-1}L_\lambda) \tilde v 
   =\gamma_g^{\rho-\lambda+1} \Phi_g^* (y^{-1} L_\lambda) \tilde u.
\]
To derive the last equality we inserted the definition of $\tilde v$
and we used the analogue of \eqref{eq-transf-Llambda} for $(y^{-1}L_\lambda)$.
Hence
\[ \bv_{\rho-\lambda}v\otimes\delta(y) =\gamma_g^{\rho-\lambda+1} \Phi_g^* (\bv_{\rho-\lambda}u\otimes\delta(y)).  \]
In view of \eqref{eq-Phi-g-at-B} we have the following formula for push-forwards of densities:
\[ (\Phi_g)_*\intd b\intd y = (\gamma_g + \bigoh(y)) \big((\Phi_g|_{y=0})_*\intd b\big)\intd y. \]
By duality, 
$\Phi_g^*(w(b)\otimes\delta(y))= \gamma_g(b,0) w(b'(b,0))\otimes \delta(y)$
holds.
Summarizing we obtain
\[ 
(\bv_{\rho-\lambda}v)(kM)\otimes\delta(y)
  = e^{(\rho-\lambda)(A(g\cdot o,kM))}(\bv_{\rho-\lambda}u)(\kappa(g^{-1}k)M)\otimes\delta(y).
\]
Hence $\bv_{\rho-\lambda}:\Et_\lambda(X)\to\Dprime(B;(N^*B)^{\rho-\lambda})$  is $G$-equivariant.

\newcommand{\etalchar}[1]{$^{#1}$}

\end{document}